\documentclass{article}
\usepackage{amsthm, amsmath, amssymb, inputenc, authblk, comment, tikz, array, longtable, enumerate, comment, xcolor}

\newtheorem{theorem}{Theorem}[section]
\newtheorem{conjecture}[theorem]{Conjecture}
\newtheorem{claim}{}[theorem]
\newtheorem{lemma}[theorem]{Lemma}
\newtheorem{problem}[theorem]{Problem}
\theoremstyle{definition}

\DeclareMathOperator{\ex}{ex}

\tikzset{every node/.style={inner sep=1.5pt,circle,draw,fill}}

\tikzstyle{every path}=[thick]

\newdimen\R
\newcommand{\cP}{\mathcal P}

\title{Dense circuit graphs and the planar Tur\'an number of a cycle}

\author[1]{Ruilin Shi}
\author[1]{Zach Walsh}
\author[1]{Xingxing Yu}

\affil[1]{School of Mathematics, Georgia Institute of Technology, USA}

\date{\today}

\begin{document}

\maketitle

\begin{abstract}
The \emph{planar Tur\'an number} $\textrm{ex}_{\mathcal P}(n,H)$ of a graph $H$ is the maximum number of edges in an $n$-vertex planar graph without $H$ as a subgraph. Let $C_k$ denote the cycle of length $k$. The planar Tur\'an number $\textrm{ex}_{\mathcal P}(n,C_k)$ is known for $k\le 7$. 
We show that dense planar graphs with a certain connectivity property (known as circuit graphs) contain large near triangulations,  and we use this result to obtain consequences for planar Tur\'an numbers.
In particular, we prove that there is a constant $D$ so that $\textrm{ex}_{\mathcal P}(n,C_k) \le 3n - 6 - Dn/k^{\log_2^3}$ for all $k, n\ge 4$.
When $k \ge 11$ this bound is tight up to the constant $D$ and proves a conjecture of Cranston, Lidick\'y, Liu, and Shantanam.
\end{abstract}

\section{Introduction}
The \emph{Tur\'an number} $\ex(n,H)$ of a graph $H$ is the maximum number of edges in an $n$-vertex graph without $H$ as a subgraph.
Tur\'an's theorem \cite{Turan}, a cornerstone of extremal graph theory, determines $\ex(n, K_t)$. 
This has led to a huge amount of related work for graphs and hypergraphs; see  \cite{Erdos-Stone}, \cite{ALON2016146}, and \cite{Keevash}.

A well-studied variant of Tur\'an's theorem involves restrictions of the host graph to an interesting class of graphs, such as hypercubes (see \cite{hypercube}) or  Erd\H os-R\'enyi random graphs (see \cite{random}).  
In 2015, Dowden \cite{Dowden2015ExtremalCP} considered planar graphs, and defined the \emph{planar Tur\'an number} $\ex_{\cP}(n,H)$ of a graph $H$ to be the maximum number of edges in an $n$-vertex planar graph without $H$ as a subgraph.  This initiated a flurry of research on planar Tur\'an numbers; see the survey paper of Lan, Shi, and Song \cite{Survey} for details regarding recent related work. 

In this paper, we focus on the planar Tur\'an number of $C_{k}$, the cycle of length $k$. Euler's formula easily gives $\ex_{\cP}(n, C_3) = 2n - 4$ for all $n \ge 3$. Dowden proved that $\ex_{\cP}(n, C_4) \le \frac{15(n-2)}{7}$ for all $n \ge 4$ and $\ex_{\cP}(n, C_5) \le \frac{12n - 33}{5}$ for all $n \ge 11$, and showed that in both cases equality holds infinitely often \cite{Dowden2015ExtremalCP}.
In \cite{6-cycle}, Ghosh, Gy\H ori, Martin, Paulos, and Xiao took the next step and proved that $\ex_{\cP}(n, C_6) \le \frac{5n}{2} - 7$ for all $n \ge 18$ with equality holding infinitely often, improving upon a result of Yan, Shi, and Song \cite{Theta-free}.
They also made the following conjecture.

\begin{conjecture}[\cite{6-cycle}] \label{conj: false conjecture}
    $\ex_{\cP}(n, C_{k}) \le \frac{3(k - 1)}{k}n - \frac{6(k + 1)}{k}$ for all $k \ge 7$ and all sufficiently large $n$.
\end{conjecture}
It was recently shown independently by Gy\H ori, Li, and Zhou \cite{7-cycle} and the present authors \cite{Shi-Walsh-Yu}  that Conjecture \ref{conj: false conjecture} holds for $k = 7$.
However, the conjecture was disproved for all $k \ge 11$ by Cranston, Lidick\'y, Liu, and Shantanam \cite{CLLScounterexample}, and later also by Lan and Song \cite{ImprovedLowerBound} and Gy{\H o}ri, Varga, and Zhu \cite{2-sum-counterexample}, using the fact that planar triangulations with at least $11$ vertices need not be Hamiltonian. Specifically, Moon and Moser \cite{Moon-Moser} showed that for each $k \ge 11$ there exists a $C_{k}$-free planar triangulation with $(2k/7)^{\log_2 3}$ vertices; this is tight up to a constant factor by a result of Chen and Yu \cite{ChenYu2002}.

We next describe the construction that disproves Conjecture \ref{conj: false conjecture}, following Cranston et al. \cite{CLLScounterexample}. For a graph $G$, we use $V(G)$ and $E(G)$ to denote its vertex set and edge set, respectively, and we write $v(G)$ for $|V(G)|$ and $e(G)$ for $|E(G)|$. 
Let $k$ and $n$ be integers with $k \ge 5$ and $n \ge k+1$.
Let  $G$ be an $n$-vertex planar graph with girth $k + 1$,  each vertex having degree $2$ or $3$, and $\frac{k+1}{k - 1}(n - 2)$ edges; such a graph exists for infinitely many integers $n$ \cite[Lemma 2]{CLLScounterexample}.
Let $G'$ be obtained from $G$ by sudividing each edge and then \emph{substituting} for each vertex of $G$ a $C_k$-free planar triangulation $B$ on as many vertices as possible, which means that each vertex $v$ of $G$ is replaced by a copy of $B$ and that $\deg_G(v)$ vertices of $B$ on a facial triangle are identified with the neighbors of $v$ in $G$.
Then $G'$ is a $C_{k}$-free graph.
Cranston et al. conjecture that $G'$ has $\ex_{\cP}(v(G'), C_k)$ edges \cite{CLLScounterexample}, but acknowledge that this will be difficult to prove, in part because the number of vertices of $B$ is only known up to a constant factor.
Because of this, they also made the following weaker conjecture.

\begin{conjecture}[\cite{CLLScounterexample}]
    There is a constant $D$ so that for all $k$ and sufficiently large $n$ we have $\ex_{\cP}(n, C_k) \le 3n - \frac{Dn}{k^{\log_2 3}}$.
\end{conjecture}

While this might not provide an exact upper bound, it is tight up to the constant $D$ when $k \ge 11$.
This follows from Cranston et al.'s construction, and also a more recent simpler construction of Gy{\H o}ri, Varga, and Zhu \cite{2-sum-counterexample} which shows that $D$ can be at most $12$.
We prove this conjecture with $D = 1/4$.

\begin{theorem} \label{main}
    $\ex_{\cP}(n, C_k) \le 3n - 6 - \frac{n}{4k^{\log_2 3}}$ for all $k,n \ge 4$.
\end{theorem}

To prove this, we first show that it suffices to consider graphs that are close to being $3$-connected.
A \emph{circuit graph} is a pair $(G, C)$ where $G$ is a $2$-connected plane graph and $C$ is a facial cycle of $G$ so that for any $2$-cut $S$ of $G$, each component of $G - S$ contains a vertex of $C$.
We show that if there is a counterexample to Theorem \ref{main}, then there is a circuit graph counterexample.

We then show that a dense circuit graph has a large near triangulation as a subgraph, where a \emph{near triangulation} is a plane graph in which every face is bounded by a triangle except possibly the outer face. 
For a $2$-connected plane graph $G$ with outer cycle $C$ we write $m(G)$ for the number of interior edges needed to make $G$ a near triangulation, so $$m(G) = 3v(G) - 6 - e(G) - (|C| - 3).$$
We will apply the following theorem with $t = \lceil k^{\log_2 3}\rceil$.

\begin{theorem}  \label{find a triangulation}
    For all $t \ge 4$, if $(G, C)$ is a circuit graph with $v(G) \ge t$ and with outer cycle $C$ so that $m(G) < \frac{v(G) - (t-1)}{3t-7}$, then $G$ has a near triangulation subgraph $T$ with $v(T) \ge t$.
\end{theorem}

This bound is sharp, as shown by iterating the construction of Figure \ref{figure: sharp example}. 
Each iteration adds three copies of a $(t-1)$-vertex triangulation arranged in a cyclic order, which adds $3t-7$ vertices and one new interior non-triangular face of size four.

Since every near triangulation is a circuit graph we can apply the following theorem of Chen and Yu \cite{ChenYu2002}.

\begin{theorem}[\cite{ChenYu2002}]  \label{Chen-Yu}
    For all $k \ge 3$, if $(G,C)$ is a circuit graph with at least $k^{\log_2 3}$ vertices, then $G$ has a cycle of length at least $k$.
\end{theorem}

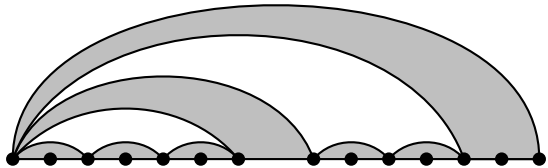
\begin{figure}
\begin{center}
\begin{tikzpicture}
\path [fill=lightgray,draw] (-4,0) to[bend left=50] (-3,0) 
    to (-4,0);
\path [fill=lightgray,draw] (-3,0) to[bend left=50] (-2,0) 
    to (-3,0);
\path [fill=lightgray,draw] (-2,0) to[bend left=50] (-1,0) 
    to (-2,0);
\path [fill=lightgray,draw] (0,0) to[bend left=50] (1,0) 
    to (0,0);
\path [fill=lightgray,draw] (1,0) to[bend left=50] (2,0) 
    to (1,0);
\path [fill=lightgray,draw] (2,0) to[bend left=50] (3,0) 
    to (2,0);
\path [fill=lightgray,draw] (-4,0) to[bend left=50] (-1,0) to (0,0)     to[bend right=70] (-4,0);
\path [fill=lightgray,draw] (-4,0) to[bend left=70] (2,0) to (3,0)     to[bend right=90] (-4,0);
     \node (1) at (-4,0) {};
     \node (2) at (-3,0) {};
     \node (3) at (-2,0) {};
     \node (4) at (-1,0) {};
     \node (5) at (0,0) {};
     \node (6) at (1,0) {};
     \node (7) at (2,0) {};
     \node (8) at (3,0) {};
     \node (10) at (-3.5,0) {};
     \node (11) at (-2.5,0) {};
     \node (12) at (-1.5,0) {};
     \node (13) at (0.5,0) {};
     \node (14) at (1.5,0) {};
     \node (15) at (2.5,0) {};
\end{tikzpicture} 
\end{center}
\caption{Each shaded region is a triangulation with $t - 1$ vertices. This graph contains no $t$-vertex near triangulation.}
 \label{figure: sharp example}
\end{figure}

Finally, it is straightforward to show (see Lemma \ref{lem: exact cycle}) that every near triangulation with a cycle of length at least $k$ also has a cycle of length exactly $k$, and therefore  Theorem \ref{main} follows from Theorems \ref{find a triangulation} and \ref{Chen-Yu}.

In fact, we will prove a result that is stronger than Theorem \ref{main}.
For each integer $k \ge 4$, we write $\Theta_k$ for the family of graphs (called \emph{theta graphs}) obtained from $C_k$ by adding a single edge, and we write $\theta_k$ for the graph obtained from $C_k$ by adding an edge that forms a triangle with two consecutive edges of $C_k$.
The number $\ex_{\cP}(n, \Theta_k)$ is known only when $k \in \{4,5,6\}$, by results of Lan, Shi, and Song \cite{Theta-free} and Ghosh, Gy\H ori, Paulos, Xiao, and Zamora \cite{GGMPX6theta}.
We prove the following, which again is tight up to the constant $1/4$ by the construction of Cranston et al \cite{CLLScounterexample}.

\begin{theorem} \label{main theta}
    $\ex_{\cP}(n, \theta_k) \le 3n - 6 - \frac{n}{4k^{\log_2 3}}$ for all $k,n \ge 4$.
\end{theorem}

This also holds with $\Theta_k$ in place of $\theta_k$ because $\theta_k \in \Theta_k$.
Interestingly, this bound may not hold if we replace $\theta_k$ with another graph in $\Theta_k$; we discuss this in more detail after Lemma \ref{lem: exact cycle}.

After describing all relevant notation, we will prove Theorem \ref{find a triangulation} in Section \ref{sec: triangulation} and then prove Theorem \ref{main theta} in Section \ref{sec: main proof}.
In Section \ref{sec: conclusion} we will discuss some consequences of Theorems \ref{main} and \ref{find a triangulation} and some related questions.

For any positive integer $k$, we let $[k]=\{1, \ldots, k\}$.   
It is well known that if $G$ is a 2-connected plane graph then each face of $G$ is bounded by a cycle of $G$ called a \emph{facial cycle}. 
The {\it interior} of a cycle $C$ in a plane graph is defined to be the subgraph of $G$ consisting of all edges and vertices of $G$ contained in the closed disc of the plane bounded by $C$.  
A path or cycle will be represented as a sequence of vertices such that consecutive vertices in the sequence are adjacent. 
For instance, $x_1x_2x_3\ldots x_k$ represents a path of length $k-1$ in which $x_ix_{i+1}$, $i\in [k-1]$, are the edges of the path, and $x_1x_2x_3\ldots x_kx_1$ represents a cycle of length $k$ with edges $x_kx_1$ and $x_ix_{i+1}$, $i\in [k-1]$. 
 For any distinct vertices $x,y$ in a graph $G$, 
 if $x,y$ are on a path $P$ then $xPy$ denotes the subpath of $P$ between $x$ and $y$. 
 For any cycle $C$ in a plane graph and any two distinct vertices $x,y$ on $C$, we use $xCy$ to denote the subpath of $C$ from $x$ to $y$ in clockwise order.

\section{Finding a large near triangulation} \label{sec: triangulation}

In this section we will prove Theorem \ref{find a triangulation}.

\begin{proof}[Proof of Theorem \ref{find a triangulation}]
Suppose that the theorem statement is false, and let $G$ be a counterexample with $m(G)$ as small as possible.
This means that $v(G) \ge t$ and $m(G) < \frac{v(G) - (t-1)}{3t-7}$ but $G$ contains no near triangulation on at least $t$ vertices.
So $m(G) > 0$, as otherwise $G$ is a near triangulation on at least $t$ vertices, a contradiction.
Moreover, $G$ is not a cycle as otherwise $m(G) = v(G) - 3$, contradicting the assumption that $m(G) < \frac{v(G) - (t-1)}{3t-7}$.
Note that if $C'$ is a cycle in $G$ and $H$ is the interior of $C'$, then $(H, C')$ is also a circuit graph.
Step-by-step, we will show that the behavior of the interior non-triangular faces of $G$ is quite restricted.

\begin{claim}
    If $F$ is a non-triangular interior facial cycle of $G$, then $F \cap C$ is either empty or is a path.
\end{claim}
\begin{proof}
Suppose that there is a non-triangular interior facial cycle $F$ so that $F \cap C$ has more than one component.
Note that $m(G) \ge |F| - 3 $.
Let $x_1, x_2, \dots, x_s$ be the vertices in $F \cap C$ listed in clockwise order on $C$.
Note that $s \ge 2$ since $F \cap C$ has more than one component.
For each $i \in [s]$, let $C_i' = x_i C x_{i+1} \cup x_i F x_{i+1}$, where $x_{s+1} = x_1$. 
If $C_i'$ is not a cycle then let $T_i=C_i'$; in this case $T_i$ has just one edge.
If $C_i'$ is a cycle then let $T_i$ be the interior of $C_i'$. 
Let $J \subseteq [s]$ so that $C_i'$ is a cycle if and only if $i \in J$.
Then $J$ is non-empty, as $G$ is not a cycle.
Note that $(T_i, C_i')$ is a circuit graph for all $i \in J$, and that $m(T_i) < m(G)$ for all $i \in J$ because $F$ is not a facial cycle of $T_i$.
We will argue that some $T_i$ with $i \in J$ contradicts the choice of $G$.

We first claim that there is some $i \in J$ so that $v(T_i) \ge t$.
Suppose not.
Then $v(G) \le s(t-2)$.
We will show that 
\begin{align}
    (3t-7) \cdot (|F| - 3) + (t-1) \ge s(t - 2).
\end{align}
If $s \le 4$, then 
$$(3t-7)\cdot (|F| - 3) + (t-1) \ge 4(t-2) \ge s(t - 2),$$
and if $s \ge 5$, then 
$$(3t-7)\cdot (|F| - 3) + (t-1) \ge (3t-7)(s - 3) + (t-1) \ge s(t-2).$$
So (1) holds.
Then 
$$(3t-7) \cdot m(G) + (t-1) \ge (3t-7)\cdot (|F| - 3) +(t-1) \ge  v(G),$$
a contradiction.
So $v(T_i) \ge t$ for some $i \in J$.

Let $J' \subseteq J$ so that $v(T_i) \ge t$ if and only if $i \in J'$. 
Note that $m(T_i) > 0$ for all $i \in J'$ or else $T_i$ is a near triangulation with at least $t$ vertices, a contradiction.
Also note that 
\begin{align}
    m(G) &= (|F| - 3) + \sum_{i \in [s]}{m(T_i)} \ge (|F| - 3) + \sum_{i \in J'}{m(T_i)}  \\
    v(G) &\le \Big(\sum_{i \in [s]}{v(T_i)}\Big) - s \le (s - |J'|)(t - 1)  +  \Big(\sum_{i \in J'}{v(T_i)}\Big)  - s,
\end{align}
where we subtract $s$ in line (3) because $x_i$ is counted twice for each $i \in [s]$.
Choose $j \in J'$ so that $\frac{v(T_j) - (t-1)}{m(T_j)}$ is maximized; this is well-defined because $m(T_j) > 0$ when $j \in J'$.
Then we have
\begin{align}
    \frac{v(T_j) - (t-1) }{m(T_j)} &\ge \frac{\sum_{i \in J'} \big(v(T_i) - (t-1)\big)}{\sum_{i \in J'} m(T_i)} \\
    &\ge \frac{\Big(\sum_{i \in J'} v(T_i)\Big) - |J'|(t-1)}{\sum_{i \in J'} m(T_i)} \\
    &\ge \frac{v(G) - |J'|(t-1) - (s-|J'|)(t - 1) + s}{m(G) - (|F| - 3)} \\
    &= \frac{\big(v(G) - (t-1)\big) - s(t - 2) + t - 1}{m(G) - (|F| - 3)}\\
    &\ge \frac{v(G) - (t-1)}{m(G)},
\end{align}
where line (6) follows from (2) and (3), and line (8) follows from (1) and the assumption that $\frac{v(G) - (t-1)}{m(G)} > 3t-7$.
Hence $\frac{v(T_j) - (t-1)}{m(T_j)} > 3t-7$.
But since $m(T_j) < m(G)$ and $v(T_j) \ge t$, this contradicts that $m(G)$ is as small as possible.
\end{proof}

Now every interior non-triangular face is either disjoint from $C$ or meets $C$ in a path.
We next show that such a path has no edges.
Recall the assumption that $v(G) > (3t-7) \cdot m(G) + (t-1)$.

\begin{claim}
    No interior non-triangular face is incident with an edge of $C$.
\end{claim}
\begin{proof}
Suppose there is an interior non-triangular facial cycle $F$ so that $F \cap C$ has at least one edge.
Then $F \cap C$ is a subpath $P$ of $C$ for which each internal vertex has degree 2 in $G$.
Let $r$ be the number of internal vertices of $P$, and let $H$ be the subgraph of $G$ obtained from deleting all internal vertices of $P$, or deleting the edge of $P$ if $r = 0$.
Note that the outer face of $H$ is bounded by a cycle $C'$ since $F \cap C$ is a path, and therefore $(H, C')$ is a circuit graph.
Also, $m(H) < m(G)$ because $F$ is not a facial cycle of $H$.

We claim that $v(H) \ge t$.
Suppose not.
Since $v(G) \ge t$ this implies that $r \ge 1$.
If $r \le 4$ then $v(G) \le t + 3 \le 4(t-2)\le (3t-7) \cdot m(G) + (t-1)$, which contradicts the assumption that $(3t-7) \cdot m(G) + (t-1) < v(G)$.
So $r \ge 5$.
Using the estimate $m(G) \ge |F| - 3 \ge r - 1$, we have
\begin{align*}
    v(G) &\le t - 1 + r \\
    &\le t - 1 + r + \big(3rt - 3t - 8r + 8\big) \\
    &= (3t-7)(r - 1) + (t-1) \\
    &\le (3t-7) \cdot m(G) + (t-1),
\end{align*}
a contradiction.
Therefore $v(H) \ge t$.

Since $v(H) \ge t$, it follows that $(3t-7)\cdot m(H) + (t-1) \ge v(H)$ or else $(H, C')$ contradicts the choice of $(G, C)$ with $m(G)$ minimum.
But if $r \le 4$ then 
\begin{align*}
    (3t-7) \cdot m(H) + (t-1) &\le (3t-7) \cdot (m(G) - 1) +(t-1)\\
    &< v(G) - (3t-7) \\
    &\le v(G) - 4 \\
    &\le v(G) - r = v(H),
\end{align*}
which is a contradiction.
If $r \ge 5$ then $m(H) \le m(G) - (r - 1)$ because $F$ is not a facial cycle of $H$, and so
\begin{align*}
    (3t-7) \cdot m(H) + (t-1) &\le (3t-7) \cdot \big(m(G) - (r - 1)\big) + (t-1) \\
    &< v(G) - (3t-7)(r - 1) \\
    & = v(G) - r - \big(3tr - 8r - 3t + 7\big) \\
    &\le v(G) - r \\
    &= v(H),
\end{align*}
again, a contradiction.
\end{proof}

Now every interior non-triangular face meets $C$ in at most one vertex.
In fact, something stronger is true.

\begin{claim} \label{claim: face meets C}
    No interior non-triangular face of $G$ meets $C$.
\end{claim}
\begin{proof}
    Suppose that an interior non-triangular facial cycle $F$ of $G$ meets $C$.
    Then $|F \cap C| = 1$ by the previous two claims; let $x \in V(F \cap C)$.
    Let $y_1, y_2, \dots, y_s$, $z_1, z_2, \dots, z_r$ be the neighbors of $x$ listed in clockwise order around $x$ where $xy_1$ and $xz_r$ are edges of $C$ and $xy_s$ and $xz_1$ are edges of $F$.
    Let $Y = \{xy_i \colon i \in [s]\}$ and let $Z = \{xz_i \colon i \in [r]\}$.
    
    Since $m(G - Y) < m(G)$ (because $F$ is not a facial cycle of $G - Y$) and $v(G - Y) = v(G)$  it follows from the minimality of $m(G)$ that $G - Y$ is not $2$-connected.
    Similarly, $G - Z$ is not $2$-connected.
    Both statements imply that $G - x$ is not $2$-connected.
    Since $x$ is on $C$, every cut vertex of $G - x$ is also on $C$ because $(G, C)$ is a circuit graph.
    Since every $2$-cut of $G$ has both vertices on $C$ but $F$ has only one vertex on $C$, there is a  unique block $B$ of $G - x$ that contains $F- x$ and satisfies $B \cap (C - x) \ne \varnothing$.
    Since $G - Y$ is not $2$-connected, there is a cut vertex $v_1$  of $G - x$ in $B$. 
    Similarly, there is a cut vertex $v_2$ of $G - x$ in $B$.
    Note that $v_1, v_2$ are on $C$ and that $x, y_1, v_1, v_2, z_r$ appear in this order clockwise around $C$, by the definitions of $Y$ and $Z$.
    Also, $v_1 \ne v_2$ or else the cut $\{x, v_1\}$ contradicts that $(G, C)$ is a circuit graph.

    Let $C_B$ be the outer cycle of $B$.
    Then $C_1 := y_s C_B v_1 \cup x C v_1 \cup xy_s$ is a cycle of $G$; let $H_1$ be the interior of $C_1$.
    Similarly, $C_2:=v_2 C_B z_1 \cup v_2 C x \cup xz_1$ is a cycle of $G$; let $H_2$ be the interior of $C_2$.
    Note that $H_1$ and $H_2$ intersect only at $x$ and that each of $(H_1, C_1)$, $(H_2, C_2)$, $(B, C_B)$ is a circuit graph.
    Also, $V(G) = V(H_1) \cup V(H_2) \cup V(B)$ and $m(G) = m(H_1) + m(H_2) + m(B) + |F| - 3$, because every interior face of $G$ other than the face bounded by $F$ is an interior face of $H_1$, $H_2$, or $B$.

    Note that $\max\{v(H_1), v(H_2), v(B)\} \ge t$, or else $v(G) \le 3(t-2) \le (3t-7) \cdot m(G) + (t-1)$, a contradiction.
    Let $H_3 = B$, and let $S \subseteq \{1,2,3\}$ so that $v(H_i) \ge t$ if and only if $i \in S$.
    Then $m(H_i) > 0$ for all $i \in S$, or else $H_i$ is a near triangulation in $G$ with at least $t$ vertices.
    Choose $j \in S$ so that $\frac{v(T_j) - (t-1)}{m(T_j)}$ is maximized.
    Then
    \begin{align*}
        \frac{v(H_j) - (t-1)}{m(H_j)} &\ge \frac{\sum_{i \in S} \big(v(H_i) - (t-1)\big)}{\sum_{i \in S} m(H_i)} \\
        &\ge \frac{v(G) - 3(t-2)}{m(G) - 1} \\
        &= \frac{v(G) - (t-1) - (2t- 5)}{m(G) - 1}.
    \end{align*}
    Since $(3t - 7)\cdot m(G) < v(G) - (t-1)$ by assumption, we have
    \begin{align*}
        (3t-7)\cdot (m(G) - 1) < v(G) -(t-1) - (3t-7) \le v(G) -(t-1) - (2t-5),
    \end{align*}
    and it follows that $\frac{v(H_j) - (t-1)}{m(H_j)} > 3t-7$.
    But since $v(H_j) \ge t$ and $m(H_j) < m(G)$, this is a contradiction.
\end{proof}

We now finish the proof.
For each interior non-triangular facial cycle $F$ of $G$, choose vertices $x_1,x_2$ on $C$ and $y_1, y_2$ on $F$ and an $x_i$-$y_i$ path $P_i$ for each $i \in [2]$ so that $P_1$ and $P_2$ are vertex-disjoint and both internally disjoint from $C\cup F$, and the graph $H_F$ bounded by the cycle $C_F:=P_1 \cup y_1 F y_2 \cup P_2 \cup x_1 C x_2$ is minimal (with respect to subgraph containment).
Let $F$ be an interior non-triangular facial cycle of $G$ so that $H_F$ is minimal over all choices of $F$.
Note that $H_F$ is $2$-connected because every face of $H_F$ is bounded by a cycle, and so $(H_F, C_F)$ is a circuit graph.

Suppose $H_F$ has an interior non-triangular facial cycle $F'$. Since $H_F$ is $2$-connected, there are vertex-disjoint paths $Q_1$ and $Q_2$ in $H_F$ from $F'$ to $z_1,z_2\in V(C_F)$, respectively, that are internally disjoint from $C_F\cup F'$. Then $Q_1,Q_2$ can be extended along $C_F$ to disjoint paths from $F'$ to $x_1Cx_2$ giving rise to an $H_{F'}$ that is a proper subgraph of $H_F$, a contradiction. Therefore $H_F$ is a near triangulation.

Then $v(H_F)<t$, as $G$ contains no near triangulation on at least $t$ vertices. 
Let $G'$ be obtained from $G$ by deleting all vertices and edges of $H_F$ that are not on $P_1 \cup P_2$.
The outer face of $G'$ is bounded by the cycle $C' = P_1 \cup y_2 F y_1 \cup P_2 \cup x_2 C x_1$ and so $(G', C')$ is a circuit graph.
If $v(G') < t$, then $v(G) < 2t$ and $m(G) \ge \frac{v(G) - (t-1)}{3t-7}$, a contradiction.
So $v(G') \ge t$.
Also, $m(G') < m(G)$ because $F$ is not a face of $G'$, and it follows from $(3t-7) \cdot m(G) +(t-1) < v(G)$ and $v(H_F) <t$ that $(3t-7) \cdot m(G') + (t-1)< v(G')$.
But then $G'$ contradicts the choice of $G$.
\end{proof}

\section{The main proof} \label{sec: main proof}

In this section we will prove Theorem \ref{main theta}, which directly implies Theorem \ref{main}.
We need the following lemma.

\begin{lemma} \label{lem: exact cycle}
    If $G$ is a near triangulation with a cycle of length at least $k$, then $G$ has a $\theta_k$ subgraph.
\end{lemma}
\begin{proof}
    Let $C$ be a cycle of $G$ with length at least $k$ so that the interior of $C$ is as small as possible.
    Then $C$ bounds a near triangulation $H$.
    It is straightforward to prove by induction on $v(H)$ that either
    \begin{enumerate}[(i)]
    \item there is an edge $e = uv$ of $C$ in a facial triangle that also contains an interior vertex $x$ of $H$, or
    
    \item there are edge-disjoint subpaths of $C$, each of length 2 and in a facial triangle, which gives two non-adjacent vertices of degree 2 in $H$.
    \end{enumerate}
    If (i) holds, then the cycle obtained from $C - uv$ by adding the path $uxv$ contradicts the minimality of the interior of $C$.
    So (ii) holds.
    If $|C| > k$, then the cycle obtained from $C$ by deleting a degree 2 vertex and adding the edge between its neighbors contradicts the minimality of the interior of $C$.
    Therefore $|C| = k$, and (ii) implies that $H$ has a $\theta_k$-subgraph.
\end{proof}

We comment that Lemma \ref{lem: exact cycle} is not true for all graphs in $\Theta_k$, as shown by the graph in Figure \ref{figure: theta graphs}.
This is the reason why Theorem \ref{main theta} may not hold for all graphs in $\Theta_k$.
In particular, the graph in Figure \ref{figure: theta graphs} has no subgraph consisting of a 12-cycle with a chord joining two vertices of distance 6.

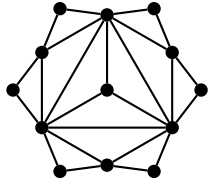
\begin{figure}
\begin{center}
 \begin{tikzpicture}
     \node[label = {}] (1) at (0:0) {};
     \node (2) at (90:\R) {};
     \node (3) at (210:\R) {};
     \node (4) at (330:\R) {};
     \node (5) at (90+60:\R) {};
     \node (6) at (90+120+60:\R) {};
     \node (7) at (90+240+60:\R) {};
     
     \node (8) at (90+30:5\R/4) {};
     \node (9) at (90+60+30:5\R/4) {};
     \node (10) at (210+30:5\R/4) {};
     \node (11) at (270+30:5\R/4) {};
     \node (12) at (330+30:5\R/4) {};
     \node (13) at (60:5\R/4) {};
\draw (1) -- (2);
\draw (2) -- (3);
\draw (3) -- (1);
\draw (1) -- (4);
\draw (2) -- (4);
\draw (3) -- (4);

\draw (2) -- (5);
\draw (2) -- (7);
\draw (3) -- (5);
\draw (3) -- (6);
\draw (4) -- (6);
\draw (4) -- (7);

\draw (8) -- (2);
\draw (8) -- (5);
\draw (9) -- (5);
\draw (9) -- (3);
\draw (10) -- (3);
\draw (10) -- (6);
\draw (11) -- (6);
\draw (11) -- (4);
\draw (12) -- (4);
\draw (12) -- (7);
\draw (13) -- (7);
\draw (13) -- (2);
\end{tikzpicture}
\end{center}
\caption{A near triangulation with a cycle of length $12$ but not every graph in $\Theta_{12}$ as a subgraph.}
 \label{figure: theta graphs}
\end{figure}

We are now ready to prove our main result.

\begin{proof}[Proof of Theorem \ref{main theta}]
    Fix $k \ge 4$, and let $t = \lceil k^{\log_2 3} \rceil$.
    We will show that $\ex_{\cP}(n, \theta_k) \le 3n - 6 - \frac{n}{4(t-2)}$, which implies that $\ex_{\cP}(n, \theta_k) \le 3n - 6 - \frac{n}{4k^{\log_2 3}}$.
    Suppose there is a $\theta_k$-free plane graph $G$ with $v(G) \ge t$ and $e(G) > 3v(G) - 6 - \frac{v(G)}{4(t-2)}$, and assume that $G$ has no proper subgraph with these properties.
    We will show that $G$ has a $\theta_k$-subgraph.
    We first reduce to the $3$-connected case.

    \begin{claim}
        $G$ is $3$-connected.
    \end{claim}
    \begin{proof}
        Suppose $G$ is not 3-connected.
        Then $G$ has edge-disjoint subgraphs $G_1$ and $G_2$ such that $G = G_1 \cup G_2$ and $|V(G_1) \cap V(G_2)| = j$ with $j \le 2$.
        By the minimality of $v(G)$, for each $i \in [2]$ either $v(G_i) < t$ or $e(G_i) \le 3v(G_i) - 6 - \frac{v(G_i)}{4(t-2)}$.
        We consider the three possibilities separately.
        
        First, assume $v(G_i) < t$ for $i\in [2]$. Then $v(G)<2t$. Therefore, since $G$ is not a triangulation,  $e(G) \le 3v(G) - 6 - \frac{v(G)}{4(t-2)}$, a contradiction.
        
        Now assume $e(G_i) \le 3v(G_i) - 6 - \frac{v(G_i)}{4(t-2)}$ for $i \in [2]$. Then $v(G_1) + v(G_2) \le v(G) + 2$, and so
        \begin{align*}
            e(G) &\le e(G_1) + e(G_2) \\
            & \le \Big( 3v(G_1) - 6 - \frac{v(G_1)}{4(t-2)} \Big)+ \Big( 3v(G_2) - 6 - \frac{v(G_2)}{4(t-2)}\Big) \\
            &\le 3v(G) - 6 - \frac{v(G)}{4(t-2)},
        \end{align*}
        a contradiction.
        
        So, up to relabeling, we may assume that $e(G_1) \le 3v(G_1) - 6 -\frac{v(G_1)}{4(t-2)}$ and $v(G_2) < t$.
        If $V(G_1) \cap V(G_2)$ is not an independent set in $G$, then $v(G_1) + v(G_2) = v(G) + 2$, and
        \begin{align*}
            e(G) &= e(G_1) + e(G_2) - 1 \\
            & \le \Big( 3v(G_1) - 6 - \frac{v(G_1)}{4(t-2)} \Big)+ \Big( 3v(G_2) - 6 \Big)- 1 \\
            &= 3v(G) - 7 - \frac{v(G_1)}{4(t-2)} \\
            &\le 3v(G) - 7 - \frac{v(G) - (t-2)}{4(t-2)} \\
            &= 3v(G) - 6 - \frac{v(G)}{4(t-2)} - \Big(1 - \frac{t-2}{4(t-2)}\Big),
        \end{align*}
        a contradiction.
        If $V(G_1) \cap V(G_2)$ is independent in $G$, then either $G_2$ is not a triangulation or $v(G_1) + v(G_2) \le v(G) + 1$.
        In either case a very similar calculation applies.
        So in all cases we arrive at a contradiction, and therefore $G$ is $3$-connected.
    \end{proof}

    Let $C$ be the outer cycle of $G$, and note that $(G, C)$ a circuit graph.
    Since $e(G) > 3v(G) - 6 - \frac{v(G)}{4(t-2)}$ we have $m(G) < \frac{v(G)}{4(t-2)} \le \frac{v(G) - (t-1)}{3t-7}$.
    By applying Theorem \ref{find a triangulation} to $(G, C)$ with $t = \lceil k^{\log_2 3}\rceil$, $G$ has a near triangulation subgraph $T$ with $v(T) \ge k^{\log_2 3}$.
    Since every near triangulation is a circuit graph, Theorem \ref{Chen-Yu} and Lemma \ref{lem: exact cycle} imply that $T$ contains $\theta_k$ as a subgraph.
\end{proof}

\section{Related problems} \label{sec: conclusion}

Theorem \ref{main} has interesting consequences related to circumference.
The \emph{circumference} of a graph (with a cycle) is the length of a longest cycle in the graph.
Erd\H os and Gallai proved that every $n$-vertex graph with circumference less than $k$ has at most $\frac{(n-1)(k-1)}{2}$ edges, and this is tight for infinitely many integers $n$ \cite{ErdosGallai}.
Theorem \ref{main} gives an analogue for planar graphs, because graphs with circumference less than $k$ are certainly $C_k$-free.

\begin{theorem}
    Let $n,k \ge 4$. If $G$ is an $n$-vertex planar graph with circumference less than $k$, then $e(G) \le 3n - 6 - \frac{n}{4k^{\log_2 3}}$.
\end{theorem}

Again, this is tight up to the constant $1/4$ because the construction of Gy{\H o}ri, Varga, and Zhu \cite{2-sum-counterexample} gives a lower bound of $3n - 6 - \frac{12n}{k^{\log_2 3}}$.

We believe that analogues of the Erd\H os-Gallai theorem are interesting for other classes of graphs (such as graphs with bounded genus, or without a fixed subgraph $H$), and also more generally for classes of matroids.
The \emph{circumference} of a matroid (with a circuit) is the number of elements in a largest circuit in the matroid.
We mention one question in this direction that is closely related to Theorem \ref{main}.
A matroid is \emph{cographic} if it is the dual of the cycle matroid of a graph.
It is well-known that a rank-$n$ cographic matroid has at most $3n - 3$ elements; this bound is tight for planar-graphic matroids.

\begin{problem}
    Show that there is a constant $D$ so that for all $n, k \ge 4$, if $M$ is a rank-$n$ cographic matroid with circumference less than $k$, then the number of elements of $M$ is at most $3n - 3 - \frac{Dn}{k^{\log_2 3}}$.
\end{problem}

This would follow from Theorem \ref{main} with $D = 1/4$ if every rank-$n$ cographic matroid with circumference less than $k$ and greater than $3n - 3 - \frac{Dn}{k^{\log_2 3}}$ elements is in fact planar-graphic.

Finally, we point out that Theorem \ref{find a triangulation} may be true for all $2$-connected plane graphs, not just circuit graphs.

\begin{problem}
    Does Theorem \ref{find a triangulation} hold for all $2$-connected plane graphs?
\end{problem}

The structure of 2-cuts in plane graphs that are not circuit graphs can be quite complicated, and it is unclear how to reduce from general 2-connected plane graphs to circuit graphs.
However, while we do make use of the connectivity properties of circuit graphs in Claim \ref{claim: face meets C}, there may be other arguments that hold for all $2$-connected plane graphs.

\bibliographystyle{plain}
\bibliography{references.bib}

\end{document}